\documentclass[a4paper,12pt]{amsart}
\usepackage{latexsym,amssymb,amsmath,amsopn,graphics,xy,epsfig,picture,epic}
\usepackage{color}

\def\aa{{\mathcal A}}

\def\cc{{\mathcal C}}

\def\ee{{\mathcal E}}

\def\mm{{\mathcal M}}

\def\rr{R}

\def\ffi{\varphi}

\def\dst{\displaystyle}

\newcommand{\HUP}{Heisenberg uni\-que\-ness pair}

\DeclareMathOperator{\vect}{span}

\def\N{{\mathbb{N}}}

\def\Q{{\mathbb{Q}}}
\def\R{{\mathbb{R}}}
\def\S{{\mathbb{S}}}

\def\Z{{\mathbb{Z}}}

\newcommand{\scal}[1]{{\left\langle{#1}\right\rangle}}

\newtheorem{lemma}{Lemma}[section]
\newtheorem{proposition}[lemma]{Proposition}
\newtheorem{theorem}[lemma]{Theorem}
\newtheorem{corollary}[lemma]{Corollary}

\theoremstyle{definition}

\newtheorem{remark}{Remark}
\newtheorem{example}{Example}

\newtheorem{definition}{Definition}

\date{\today}

\begin{document}
\title[The Cramer-Wold Theorem on Quadratic Surfaces]{The Cramer-Wold Theorem on Quadratic Surfaces and Heisenberg
  Uniqueness Pairs}

\author{Karlheinz Gr\"ochenig}
\address{Faculty of Mathematics \\
University of Vienna \\
Nordbergstrasse 15 \\
A-1090 Vienna, Austria}
\email{karlheinz.groechenig@univie.ac.at}

\author{Philippe Jaming}
\address{Univ. Bordeaux, IMB, UMR 5251, F-33400 Talence, France.
CNRS, IMB, UMR 5251, F-33400 Talence, France.}
\email{Philippe.Jaming@math.u-bordeaux.fr}

% \thanks{K.\ G.\ was
%   supported in part by the  project P26273 - N25  of the
% Austrian Science Fund (FWF)}

\begin{abstract}
Two measurable sets $S, \Lambda \subseteq \R ^d$ form a \HUP ,
if every bounded measure $\mu $ with support in $S$ whose Fourier
transform vanishes on $\Lambda $ must be zero. We show that a
quadratic hypersurface and the union of two  hyperplanes in general position form a
\HUP\ in $\R ^d$. As a corollary we obtain a new, surprising version
of the classical Cram\'er-Wold theorem:  a bounded measure
supported on a quadratic hypersurface is uniquely determined by its
projections onto two generic hyperplanes (whereas an arbitrary measure
requires the knowledge of a dense set of projections). 
We also give an  application to the  unique continuation of
eigenfunctions of second-order  PDEs with constant coefficients.
\end{abstract}



\maketitle

%\tableofcontents

\section{Introduction}

The notion of a \HUP\  by Hedenmalm and Montes-Rodr\'igez~\cite{HMR}
introduced a new facet to the investigation of Fourier transform pairs $(\mu ,
\widehat{\mu })$ for a bounded Borel measure $\mu $ and its Fourier
transform (or characteristic function) 
$$
\widehat{\mu}(\xi)=\int_{\R ^d}
e^{i\scal{x,\xi}}\,\mbox{d}\mu(x),\qquad\xi\in\R^d \, .
$$
% Fourier transform pairs are the realm of the uncertainty principle.
In general a version of the  uncertainty principle balances the size (of the support or
of some quantity) of $\mu $ with the size of $\widehat{\mu }$, so that
both cannot be too small. By contrast, \HUP s balance the size of the
support of a measure $\mu $ with the size of the zero set of
$\widehat{\mu }$. A particular aspect is that both  the support of
$\mu$  and the zero
set of $\widehat{\mu }$ may be singular sets.

\subsection{Heisenberg Uniqueness Pairs}\ 

To be precise, let $S,\Lambda \subseteq \R ^d$ be two measurable sets and
let $\mm(S)$ denote  the set of \emph{finite} signed Borel measures supported in $S$.
The following definition was first introduced in
\cite{HMR} in a sightly more restrictive form.

\begin{definition} 
The pair $(S,\Lambda) \subset \R ^d \times \R ^d$ is said to be a
\emph{Heisenberg uniqueness pair},  %(HUP)
if the only measure $\mu\in\mm(S)$ such that $\widehat{\mu}=0$ on $\Lambda$ is the measure $\mu=0$.

For a subset   $\cc\subset\mm(S)$ of measures,   $(S,\Lambda)$ is said to be a \emph{$\cc$-Heisenberg uniqueness pair}
if the only measure $\mu\in\cc$ such that $\widehat{\mu}=0$ on $\Lambda$ is the measure $\mu=0$.
\end{definition}

Equivalently, if two finite measures supported on $S$ have characteristic functions agreeing on $\Lambda$, then
they are equal.

If $S$ is a smooth manifold, we write $\aa\cc(S)$ for the set of finite measures that are absolutely continuous with respect 
to the surface measure on $S$. The original concept in~\cite{HMR} was that of
an  $\aa\cc(S)$-Heisenberg uniqueness pair
and  so far was almost exclusively  studied in dimension $d=2$  when  $S$ is a smooth curve.

Since a general characterization of \HUP s is out of reach, research
so far has focussed on the investigation of specific examples, for instance
$S$ being a hyperbola~\cite{HMR}, a
circle~\cite{Le,Sj1}, a parabola~\cite{Sj2}, and $\Lambda $ a set of lines
or a discrete subset thereof, some cases with three parallel  lines are treated
in~\cite{Ba},  the case of both $S$ and $\Lambda $
being circles or spheres is treated in \cite{Le,Sj1,Sri}.

In \cite{JK} most of these examples were unified and extended by means
of a new technique. The question of  whether $(S,\ell_1\cup\ell_2)$ is a
$\aa\cc(S)$-Heisenberg uniqueness pair  can then be answered by
studying  a certain dynamical system on $S$
defined by two (distinct) lines $\Lambda = \ell _1 \cup \ell _2$.\footnote{Througout the paper, two lines will always mean two {\em distinct} lines.}
With this technique, one can derive a complete characterization in the
case of conic sections  $S=\{(x,y)\in\R^2\,:
ax^2+bxy+cy^2+dx+ey=f\} \subseteq \mathbb{R}^2$ with $a,b,c,d,e,f\in\R$ not all $0$ and two
lines  $\Lambda = \ell _1 \cup \ell _2$. Before a description of these
results, we make a convenient reduction. 
As already observed in \cite{HMR}, the notion of Heisenberg uniqueness pairs
is invariant with respect to affine linear transformations, namely: 
% inherits the invariance properties of the Fourier transform with
% respect to 
% affine transformations, namely:
\begin{itemize}
\item[{\em [Inv 1]}] Fix $x, \xi \in\R^d$. Then
  $\bigl(S,\Lambda\bigr)$ is a \HUP , if and only if 
$\bigl(S-x,\Lambda-\xi \bigr)$ is a \HUP .

\item[{\em [Inv 2]}] Fix $T$ a linear invertible transformation
  $\R^d\to\R^d$ and denote by $T^*$ its adjoint. 
Then $\bigl(S,\Lambda\bigr)$ is a \HUP ,  if and only if $\bigl(T^{-1}(S),T^*(\Lambda)\bigr)$
is a \HUP .
\end{itemize}

According to the invariance properties {\sl [Inv1-Inv2]} it is enough to consider the
following cases for the analysis of conic sections in $\R ^2$. 

% The focus of \cite{JK} was on $\Lambda$ being a set of two distinct lines through the origin $\Lambda=\ell_1\cup\ell_2$.
%  Using that technique, many examples have been given in \cite{JK}.
% In particular, one can give a complete characterization in the case $\Gamma=\{(x,y)\in\R^2\,:
% ax^2+bxy+cy^2+dx+ey=f\}$ with $a,b,c,d,e,f\in\R$ not all $0$. 
% According to the invariance properties {\sl [Inv1-Inv2]} it is enough to consider the
% following cases. 

\begin{enumerate}
\renewcommand{\theenumi}{\roman{enumi}}
\item {\bf A single line.} If $\Gamma=\{(x,y)\in\R^2\,: y=0\}$ and
  $\ell_1,\ell_2$ are two arbitrary  lines through $(0,0)$, then
$(\Gamma,\ell_1\cup\ell_2)$ is a Heisenberg uniqueness pair. Actually
$(\Gamma,\ell_1)$ is already a Heisenberg uniqueness pair 
unless $\ell_1$ is  the line orthogonal to $\Gamma$~\cite{HMR}.

\item {\bf The parabola.} If $\Gamma=\{(x,y)\in\R^2\,:\ y=x^2\}$ and
  $\ell_1,\ell_2$ are  two arbitrary  lines through $(0,0)$, then
$(\Gamma,\ell_1\cup\ell_2)$ is a Heisenberg uniqueness pair~\cite{Le,Sj2}.

\item {\bf The circle.} If $\Gamma=\{(x,y)\in\R^2\,:\ x^2+y^2=1\}$ and
  $\ell_1,\ell_2$ are  two arbitrary  lines through $(0,0)$
which intersect with an angle $\theta\notin\pi\Q$, then
$(\Gamma,\ell_1\cup\ell_2)$ is a Heisenberg uniqueness pair~\cite{Le,Sj1}. \footnote{Actually, in \cite{JK}, it is shown that
$(\Gamma,\ell_1\cup\ell_2)$ is an $\aa\cc(\Gamma)$-Heisenberg uniqueness pairs. We will give a slightly simpler proof
below that allows to extend the result.}

\item {\bf The hyperbola.} If $\Gamma=\{(x,y)\in\R^2\,:\ xy=1\}$
  and $\ell_1,\ell_2$ are two lines through $(0,0)$.  If
  $\ell_1=\R(a,b)$ and $\ell_2\not=\R(-a,b)$, 
	then
$(\Gamma,\ell_1\cup\ell_2)$ is a Heisenberg uniqueness pair~\cite{HMR}.

% \item {\bf The hyperbola.} If $\Gamma=\{(x,y)\in\R^2\,:\ xy=1\}$
%   and $\ell_1,\ell_2$ are two lines through $(0,0)$. Assume that if $\ell_1=\R(a,b)$ then $\ell_2\not=\R(-a,b)$,
% that is	$\ell_1,\ell_2$ are not the symmetric with respect to the axes of one an other, 
% 	then
% $(\Gamma,\ell_1\cup\ell_2)$ is a Heisenberg uniqueness pair~\cite{HMR}.
\end{enumerate}

% Note that the parabola and the circle cases were proved previously by P. Sj\"olin \cite{Sj1,Sj2} and N. Lev \cite{Le}.
Further degenerate quadratic curves which where not discussed in
\cite{JK}, but for which the 
techniques of \cite{JK} apply are the following. 
\begin{enumerate}
\renewcommand{\theenumi}{\roman{enumi}}
\item[(v)] {\bf The cone $\Gamma=\{(x,y)\in\R^2\,:\ xy=0\}$}.  If
  $\ell_1=\R(a,b)$ and $\ell_2\not=\R(-a,b)$, then
$(\Gamma,\ell_1\cup\ell_2)$ is a Heisenberg uniqueness pair.

% \item[(v)] {\bf The cone $\Gamma=\{(x,y)\in\R^2\,:\ xy=0\}$}. If
%   $\ell_1,\ell_2$ are two lines through $(0,0)$, that are not
%   symmetric  with respect to the axes of one another,	then
% $(\Gamma,\ell_1\cup\ell_2)$ is a Heisenberg uniqueness pair.

\item[(vi)] {\bf Two parallel lines  $\Gamma=\{(x,y)\in\R^2\,:\
    x^2=a^2\}$}. If $\ell_1,\ell_2$ are two arbitrary lines through $(0,0)$
except the $y$ axis, $(\Gamma,\ell_1\cup\ell_2)$ is a Heisenberg uniqueness pair.
\end{enumerate}

By  applying the invariance properties of Heisenberg uniqueness pairs,
we may summarize these results as follows.
% see that the following theorem has been established 
% previously. 

\begin{theorem}\label{th:dim2}
Let $Q$ be a quadratic form on $\R^2$, $v\in\R^2$, $\rho\in\R$, and
$S=\{(x,y)\in\R^2\,:\ Q(x,y)+2\scal{v,(x,y)}=\rho\}$. Then there
exists an exceptional  set $\ee=\ee(Q,v,\rho)$
of pairs of distinct directions, such that when $\ell_1,\ell_2$ are
two lines with directions not in $\ee$, then 
$(S,\ell_1\cup\ell_2)$ is a Heisenberg uniqueness pair. Moreover, if
$\ell _1$ is fixed, then the set $\{\ell _2: (\ell _1, \ell _2) \notin
\ee\}$ is at most countable. Therefore 
$\ee $ has measure zero with respect to the surface measure on $\S ^1
\times \S ^1$. 
\end{theorem}

\begin{remark}
The dependence of the exceptional set on the parameters $Q,v,\rho $ is
simple. For $S$ a parabola, a point (= degenerate ellipse), or two parallel lines
(= degenerate parabola), 
$\ee $ is the empty set. In all other cases
$\ee (Q,v,\rho )$ is a fixed set of directions that depends only on
$Q$, but not on $v$ and $\rho $. Thus for a fixed quadratic form  $Q$
only a set $\ee (Q)$ and $\emptyset $ may occur. 
% If $v\in\R^2$ and $Q$ is a degenerate quadratic form that is not a
% multiple of $\scal{v,(x,y)}^2$%  (equivalently, $v$ is not an
% % eigenvector of $Q$)
% , then $\Gamma$ is a parabola and thus 
% $\ee(Q,v,\rho)=\emptyset$. 

% We refrain from giving the precise description of $\ee(Q,v,\rho)$ when
% $Q$ is a non-degenerate quadratic form. Note, however,  that 
% in this case $\ee(Q,v,\rho)$ does not depend on $v$.
\end{remark}

In this paper we present the first complete study of \HUP s in higher
dimensions and will extend Theorem~\ref{th:dim2} to  higher
dimensions for the case of a quadratic hypersurface $S$ and a union of
hyperplanes $\Lambda $.
We then make the connection with the classical Cram\'er-Wold Theorem
which characterizes probability measures by their  projections onto
hyperplanes. In contrast to general measures, a measure supported on a
quadratic hypersurface is  uniquely determined already by its
projection to two generic hyperplanes. Finally, as in~\cite{HMR} we interpret the
notion of \HUP\ on a quadratic hypersurface as a statement about the
solutions to a partial differential equation of order two with
constant coefficients.

Let us now describe our results with more precision.
To start, let $S$ be a quadratic hypersurface, that is a surface of
the form $S=\{x\in\R^d\,:P(x)=0\}$ where $P$ is a polynomial of total
degree $2$. 
In the following we always write $P$ as the sum of a quadratic form
$Q$ and an affine form. This means that there exists a bilinear form
$B: \R ^d \times \R ^d \to \R $, such that $Q(x)= B(x,x)$ and a vector
$v\in \R ^d$ and $\rho \in \R $, such that $P(x) = Q(x) + 2 \langle
v,x\rangle -\rho $. 
The set $\Lambda$ will consist of two distinct  hyperplanes
$H_1,H_2$. We describe a  hyperplane by a normal vector  $u\in\R^d$
with unit norm $|u|=1$ and the offset parameter  $s\in\R$, and  define 
$$
H_{u,s}=\{x\in\R^d\,:\ \scal{x,u}=s\}\quad\mbox{and}\quad H_u=H_{u,0}.
$$

We first consider the case of two intersecting hyperplanes. Our main
result is the following. 

\begin{theorem}\label{th:dimd}
Let $Q$ be a quadratic form on $\R^d$, $v\in\R^d$, $\rho\in\R$ and
$S=\{ x\in\R^d\,:\ Q(x)+2\scal{v,x}=\rho\}$. There exists an
exceptional set $\ee=\ee(Q,v,\rho)$ 
of pairs of distinct directions such that
\begin{enumerate}
\renewcommand{\theenumi}{\roman{enumi}}
\item the set $\ee $ has measure zero with
respect to the surface measure on $\S ^{d-1} \times \S ^{d-1}$;

\item when $u_1,u_2\in\R^d$ satisfy $Q(u_1),Q(u_2)\not=0$
and $(u_1,u_2)\notin\ee$, then  $(S,H_{u_1}\cup H_{u_2})$ is a
Heisenberg uniqueness pair.  
\end{enumerate}
\end{theorem}

This result is very general and the proof gives an explicit
construction of the exceptional  set  in terms of the exceptional sets 
described  in the special cases of Theorem \ref{th:dim2}.
Also notice that Theorem~\ref{th:dimd} holds for arbitrary measures,
whereas the cited results in dimension $d=2$ were proved only for
absolutely continuous measures. 
 The proof of
Theorem~\ref{th:dimd} will be given in Section~2. We will
significantly extend  the approach with dynamical systems
~\cite{JK}.

We will then focus on some specific cases and derive  more
precise results.  We will  also consider the case 
of parallel hyperplanes. In Section~3 we study the case when $S$ is a
cone or a hyperboloid in $\R ^d$ and $\Lambda $ consists of two or
more 
hyperplanes. In Section~4 we study the case when $S$ is a sphere in $\R ^d$
and $\Lambda $ consists either of two parallel hyperplanes or several
hyperplanes with a common intersection. It is  a curious feature of the
dynamical systems approach that  infinite Coxeter groups will occur
naturally. Indeed, we will apply a deep theorem about such groups in
the proof of Theorem~\ref{th:cox}. 

 %  Let us be more precise by making a few reductions:

% Using the invariance properties {\sl [Inv1-Inv2]} and the reduction of quadratic forms, it is enough to consider polynomials of the form
% $$
% P(x_1,\ldots,x_d)=x_1^2+\cdots+x_p^2-x_{p+1}^2-\cdots-x_{p+q}^2+x_{p+q+r}+a
% $$
% with $p+q+r\leq d$, $r\in\{0,1\}$, $a\in\{-1,0,1\}$,  and with the
% understanding that if $q=0$ the terms $-x_{p+1}^2-\cdots-x_{p+q}^2$ 
% do not occur and if $r=0$ then the term $x_{p+q}$ does not occur. 

% In this paper, we will consider the following cases:

% \begin{enumerate}
% \renewcommand{\theenumi}{\roman{enumi}}
% \item {\bf The sphere}. $S=\{x\in\R^d\,:\ x_1^2+\cdots+x_d^2=1\}$.

% \item {\bf The cone}. $S=\{x\in\R^d\,:\ x_1^2+\cdots+x_p^2-x_{p+1}^2-\cdots-x_{p+q}^2=0\}$, $p+q=d$.

% \item {\bf The hyperboloid}. $S=\{x\in\R^d\,:\ x_1^2+\cdots+x_p^2-x_{p+1}^2-\cdots-x_{p+q}^2=1\}$, $p+q=d$.

% \item {\bf The paraboloid}. $S=\{x\in\R^d\,:\ x_1^2+\cdots+x_{d-1}^2-2x_{d}=0\}$.
% \end{enumerate}

% We want to prove that $(S,H_{u,s}\cup H_{v,t})$ is a Heisenberg uniqueness pair. There will be two cases:

% --- either $H_{u,s}\cap H_{v,t}=\emptyset$, so that $u=\pm v$. In this case, we can translate both hyperplanes
% so that we can assume that $u=v$ and $s=0$.

% --- or $H_{u,s}\cap H_{v,t}\not=\emptyset$. In this case, we can translate both hyperplanes so that 
% they intersect in a line through $0$ and then $s=t=0$.

Let us now see how our results can be interpreted in probability theory and PDEs.

\subsection{A  Cram\'er-Wold type theorem}\ 

Let $X,Y$ be two metric spaces and $f\,:X\to Y$ a Borel mapping.
Recall that, if $\mu$ is a measure on $X$, the push-forward of $\mu$ by $f$
is the measure $f\, _* \, \mu$ on $Y$ defined by $f\,_*\,\mu(E)=\mu\bigl(f^{-1}(E)\bigr)$ and that,
for every continuous compactly supported function $g$ on $Y$,
\begin{equation}
  \label{eq:c1}
\int_Y g(y)\,\mbox{d} f\,_*\,\mu(y)=\int_Xg\bigl(f(x)\bigr)\,\mbox{d}\mu(x).  
\end{equation}
The classical Cram\'er-Wold Theorem \cite{CW} asserts that a probability measure on $\R^d$ is uniquely determined by the
set 
$$
\{\pi_*\mu\,:\pi\mbox{ a projection on a hyperplane through }0\},
$$
that is, if $\pi_*\mu=\pi_*\nu$ for every $\pi$, or
equivalently $\pi_*(\mu - \nu)=0$,
then $\mu - \nu=0$. This fact is easily proven as follows: if $\pi$ is the projection on a hyperplane
$H$, then  the Fourier transform of $\pi_*(\mu-\nu)$ in $H$ is just
the  restriction of  $\widehat{\mu-\nu}$ to $H$
({\it see} \eqref{fund2} in Lemma \ref{lem:fund}). Therefore,
$\pi_*(\mu -\nu)=0$ for every $H$ if and only if $\widehat{\mu-\nu}=0$
which occurs if and only if $\mu-\nu=0$. As $\widehat{\mu-\nu}$ is continuous one only needs a dense set of hyperplanes.

In general, not much more can be said, {\it see e.g.}
\cite{BMR,Gi,He,Re}. In particular, finitely many projections  do never
determine a measure on $\R ^d$ completely. However, if we restrict the support of
a  measure 
to a quadratic hypersurface,   we obtain a   completely
different version of a Cram\'er-Wold theorem. The following statement
is an immediate consequence of Theorem~\ref{th:dimd}.

\begin{theorem}\label{th:CW}
Let $Q$ be a quadratic form on $\R^d$, $v\in\R^d$, $\rho\in\R$ and
$S=\{ x\in\R^d\,:\ Q(x )+2\scal{v,x}=\rho\}$. Then there exists an
exceptional  set $\ee=\ee(Q,v,\rho)$ of measure zero in $\S ^{d-1} \times
\S ^{d-1}$ with the following property:
Let  $u_1,u_2\in\R^d$ be distinct directions, such that  $Q(u_1),Q(u_2)\not=0$
and $(u_1,u_2)\notin\ee$. 
If  $\mu,\nu\in\mm(S)$ and  $\pi_{H_{u_1}}\,_* \,
\mu=\pi_{H_{u_1}}\,_* \, \nu$
and $\pi_{H_{u_2}}\,_* \, \mu=\pi_{H_{u_2}}\,_* \, \nu$, then  $\mu=\nu$. 
\end{theorem}

% \begin{theorem}\label{th:CW}
% Let $Q$ be a quadratic form on $\R^d$, $v\in\R^d$, $\rho\in\R$ and
% $S=\{ x\in\R^d\,:\ Q(x )+2\scal{v,x}=\rho\}$. Then there exists an
% exceptional  set $\ee=\ee(Q,v,\rho)$
% of pairs of distinct directions  such that, when $u_1,u_2\in\R^d$ satisfy $Q(u_1),Q(u_2)\not=0$
% and $(u_1,u_2)\notin\ee$, then for $\mu,\nu\in\mm(S)$, $\pi_{H_1}\,*\mu=\pi_{H_1}\,*\nu$
% and $\pi_{H_2}\,*\mu=\pi_{H_2}\,*\nu$ imply $\mu=\nu$. Moreover, $\ee
% $ has measure zero in $\S ^{d-1} \times \S ^{d-1}$. 
% \end{theorem}

Thus a finite measure supported on a quadratic hypersurface is
 determined uniquely by its projections to two generic 
hyperplanes. 

\subsection{Applications to linear PDEs}\ 

Let $P$ be a quadratic polynomial on $\R ^d$ and $S = \{ x\in \R ^d:
P(x) = 0\}$. Replacing each variable $x_j$ by the partial derivative
$\tfrac{1}{i} \tfrac{\partial}{\partial x_j}$, we obtain a second
order partial differential operator $P(D)$ with constant coefficients. If $\mu $ is a finite
measure, then a basic formula about Fourier transforms states that
$$
P(D) \widehat{\mu }(\xi ) = \int _{\R ^d} e^{i \langle x,\xi
  \rangle} P(x)\,\mbox{d}\mu (x) \, .
$$
Consequently, if $\mathrm{supp}\, \mu \subseteq S$, then 
$$
P(D) \widehat{\mu } \equiv 0 \, ,
$$
and thus $\widehat{\mu }$ is a (distributional) solution of this  PDE, as
was observed in~\cite{HMR}. Conversely, if $P(D) \widehat{\mu } \equiv
0 $, then $\mathrm{supp}\, \mu \subseteq S$. Let us give an example of the kind of results we can
obtain:

\begin{example}
Let $\Delta= \sum_{j=1}^d\tfrac{\partial^2}{\partial x_j^2}$ be the
standard Laplacian on $\R^d$. 
The simplest bounded eigenfunctions  of $\Delta$ with eigenvalue
$\rho<0$ are given by $e_\lambda(x)=e^{i\scal{\lambda,x}}$ with
$|\lambda |^2 = -\rho $. 

Now assume that $|\lambda|^2=|\mu|^2=-\rho$  and that $e_\lambda(x)=e_\mu(x)$ on some hyperplane $H$.
Without loss of generality, we may assume that $H=\R^{d-1}\times\{0\}$ and then
$\lambda_j=\mu_j$ for $j=1,\ldots,d-1$. The condition $|\lambda|^2=|\mu|^2$ then shows that there are still
two possibilities, namely $\mu_d=\pm\lambda_d$. Uniqueness is
guaranteed as soon as we assume that 
$e_\lambda(x)=e_\mu(x)$ on a second hyperplane $H'$ that is not parallel to $H$.

With $P(x)=- |x|^2-\rho$, the eigenvalue equation $\Delta u = \rho u$
can be written as $P(D)u=0$. Since 
% Now $\Delta u_\lambda=\rho u_\lambda$ can be written $P(D) u=0$ where $P(x)=-4\pi^2 |x|^2-\rho$
% and
$e_\lambda=\widehat{\delta_{\lambda}}$ and  $P(\lambda)=0$, $\delta
_\lambda $  is indeed supported on $S=\{x\,:\ P(x)=0\}$.
Theorem~\ref{th:dimd} then states that, for a generic pair of intersecting hyperplanes $(H,H')$,
$e_\lambda(x)=e_\mu(x)$ on $H\cup H'$ implies $\lambda=\mu$.

Of course, the result for the simple eigenfunctions $e_\lambda$ is valid for an arbitrary
pair of intersecting hyperplanes. However, for more general eigenfunctions, this is no
longer true  and some restrictions apply to the hyperplanes for the result to be true.
\end{example}

The actual result is  much more  general:  Theorem~\ref{th:dimd}  directly
yields the following property of solutions of homogeneous second order
PDEs. 

\begin{theorem} \label{th:pde}
Let $Q$ be a quadratic form on $\R^d$, $v\in\R^d$, $\rho\in\R$ and $P(x)=Q(x)+\scal{v,x}-\rho$.
Let $u_1,u_2\in\S^{d-1}$ be such that $Q(u_1),Q(u_2)\not=0$ and $u_1,
u_2 \not \in \ee (Q,v,\rho  )$.
Let $\mu$ be a finite measure and assume that $u=\widehat{\mu }$ solves the partial differential equation $P(D) u = 0$.

If $u$ vanishes on the hyperplanes $H_{u_1}$ and $H_{u_2}$, then $u=0$.
\end{theorem}

In other words, a non-trivial eigenfunction of a second-order linear partial differential operator with constant coefficients
cannot vanish on two generic hyperplanes. It remains to be seen
whether a similar result holds for distributions instead of measures. 

  %%%By varying the parameter
                                %%%$\rho $, this result applies in
                                %%%particular to the eigenfunctions of
                                %%%a second order linear partial
                                %%%differential operator.  

\section{Quadratic Hypersurfaces and Hyperplanes}

%\section{Proof of Theorem \ref{th:dimd}}
\subsection{Measures on quadratic hypersurfaces with Fourier transform vanishing on a hyperplane}\ 

In this section we build up the proof of Theorem \ref{th:dimd} and
first  investigate measures on a quadratic hypersurface $S$ whose 
Fourier transform vanishes on a single  hyperplane. We start with some
  simple geometric observations.  

\begin{lemma}
\label{lem:reflection} %$\pi_u\,_*\,\nu_{u,s}=0$
Let $B$ be a bilinear form on $\R^d\times\R^d$ and $Q$
be the associated quadratic form $Q(x)=B(x,x)$.
Let $u,v\in\R^d$ with $|u|=1$ and let $\pi_u$ be the orthogonal projection
on the hyperplane $H_u$. Define the affine linear
  transformation $\rr _{u,v}^Q$ by 
\begin{equation}
  \label{eq:c2}
  \rr_{u,v}^Q(x):=x-2\frac{B(x,u)+\scal{v,u}}{Q(u)}u \, .
\end{equation}

Fix $x\in\R^d$ and consider the equation in $y$
\begin{equation}
  \label{eq:c2bis}
	\left\{\begin{matrix}Q(y)+2\scal{y,v}&=&Q(x)+2\scal{x,v}\\ \pi_u(y)&=&\pi_u(x)\end{matrix}\right..
\end{equation}
-- If  $Q(u)\not=0$,  then \eqref{eq:c2bis} has two solutions, $y=x$ and $y=\rr_{u,v}^Q(x)$.\\
-- If  $Q(u)=0$ and $B(x,u)+\scal{v,u}\not=0$, then $y=x$ is the only solution of \eqref{eq:c2bis}.
\end{lemma}

\textbf{Notation.} If no confusion can occur, we will simply write $\rr_{u,v}=\rr_{u,v}^Q$. Further, 
if $v=0$ we will write $\rr_u=\rr_{u,0}$. In this case, $\rr_u$ is a linear mapping and we define
$\rr_u^*$ to be its adjoint: $\scal{\rr_ux,y}=\scal{x,\rr_u^*y}$ for every $x,y\in\R^d$.
	
Note that $(\rr _{u,v}^Q) ^2 = \mathrm{Id}$ and that  for the case  $B(x,u) = \scal{x,u}$
and $v=0$,  $\rr $ is just the reflection of a vector $x$ across the
hyperplane $H_u$. In general,  $\rr $ is a reflection across the hyperplane
$H_{u}$ with respect to the quadratic form $Q(x)$ as a ``metric.''
%\end{notation}

\begin{proof} If $\pi_u(x)=\pi_u(y)$ then  $y=x+tu$ for some $t\in \R $, and thus
$$
Q(y)=Q(x)+t^2Q(u)+2tB(x,u)\quad \text{ and }  \quad \scal{y,v}=\scal{x,v}+t\scal{u,v}.
$$
The identity  $Q(y)+2\scal{y,v}=Q(x)+2\scal{x,v}$ then reads as 
$$
t^2Q(u)+2tB(x,u)+2t\scal{u,v}=0.
$$
It remains to solve for $t$. If  $Q(u)\not=0$, there are two
solutions. Either 
$t=0$ and thus  $y=x$, or  $t=\dst-2\frac{B(x,u)+\scal{v,u}}{Q(u)}$
and thus  $y=R_{u,v}(x)$.
If  $Q(u)=0$ and $B(x,u)+\scal{v,u}\not=0$, then $t=0$,  and  thus $y=x$ is the only solution.
\end{proof}

 A measure supported on a quadratic
hypersurface whose Fourier transform vanishes on a hyperplane must
possess some symmetry property. The next lemma gives a precise
formulation of these symmetries and  is a
crucial  extension of Lemma 2.1 of \cite{JK}. 

\begin{lemma}
\label{lem:fund}
Let $B$ be a bilinear form on $\R^d\times\R^d$ and $Q$
be the associated quadratic form $Q(x)=B(x,x)$.
Let $u,v\in\R^d$ with $|u|=1$,  $s>0$, $\rho \in \R $, % $\eps\in\{0,1\}$,
and  let $S=\{x\in\R^d\,:Q(x)+2\scal{x,v}=\rho\}$.
For  $\mu\in\mm(S)$ let  $\nu_{u,s}$ be the measure $\mbox{d}\nu_{u,s}=e^{is\scal{x,u}}\,\mbox{d}\mu$.
Then the following are equivalent:
\begin{enumerate}
\renewcommand{\theenumi}{\roman{enumi}}
\item\label{fund1}  $\widehat{\mu}(\xi)=0$ for every $\xi\in
  H_{u,s}$. 

\item\label{fund2} $\pi_u\,_*\,\nu_{u,s}=0$. 
\end{enumerate}
Moreover, if $Q(u)\not=0$, \eqref{fund1}-\eqref{fund2} are equivalent to the following two properties:
\begin{enumerate}
\renewcommand{\theenumi}{\roman{enumi}}
\setcounter{enumi}{2}
\item\label{fund3} $\rr _{u,v}\,_*\,\nu_{u,s}=-\nu_{u,s}$.

\item\label{fund4}  $\widehat{\mu}(\xi+su)=-
e^{-2i\frac{\scal{u,v}\scal{u,\xi}}{Q(u)}}\widehat{\mu}(\rr
_{u}^*\xi+su)$ for every $\xi\in\R^d$.
\end{enumerate}
\end{lemma}

\begin{proof} 
(i) $\, \Leftrightarrow \, $ (ii) If $\xi\in H_{u,s}$, then $\xi=\pi_u\xi+su$, and  thus
$$
\widehat{\mu}(\xi)=\int_{\R^d}e^{i\scal{x,\xi}}\,\mbox{d}\mu(x)
=\int_{\R^d}e^{i\scal{\pi_u x,\pi_u\xi}}e^{is\scal{x,u}}\,\mbox{d}\mu(x).
$$
Therefore, if we denote by $\nu$ the measure defined by
$$
\mbox{d}\nu(x) =\mbox{d}\nu _{u,s} (x)= e^{is\scal{x,u}}\,\mbox{d}\mu(x),
$$
then, according to~\eqref{eq:c1},
$$
\widehat{\mu}(\xi)=\int_{\R^d}e^{i\scal{\pi_u x,\pi_u\xi}}\,\mbox{d}\nu(x)
=\int_{H_u}e^{i\scal{y,\pi_u\xi}}\,\mbox{d}\pi_u\,_*\,\nu(y)
=\widehat{\pi_u\,_*\,\nu}(\pi_u\xi)
$$
where $\widehat{\pi_u\,_*\,\nu}$ is the characteristic function of 
the measure $\pi_u\,_*\,\nu$ defined on $H_u$ (identified with $\R^{d-1}$).

Since $\pi_u$ is a bijection from $H_{u,s}$ to $H_u$,
$\widehat{\mu}=0$ on $H_{u,s}$ if and only if $\widehat{\pi_u\,_*\,\nu}=0$ on $H_u$
that is $\pi_u\,_*\,\nu=0$. 
This establishes the equivalence of (\ref{fund1}) and (\ref{fund2}).

\smallskip

(ii) $\, \Rightarrow \, $ (iii) % In view of Lemma \ref{lem:reflection}, (\ref{fund3}) is a 
% reformulation of (\ref{fund2}). 
Let $E_u=\{x\in S\,: B(x,u)=-\scal{v,u}\}$.
If $x\in E_u$, then $\rr _{u,v}(x)=x$, consequently  $\pi_u$ is
one-to-one on $E_u$ and $\nu |_{E_u} = \rr _{u,v} \, _* \, \nu
|_{E_u}$. On the complement  $S\setminus E_u$ the projection  $\pi_u$
is two-to-one. We can therefore partition $S\setminus E_u=S_+\cup S_-$ into two Borel sets $S_+,S_-$ 
in such a way that $\pi_u$ is one-to-one on  $S_+$ and on
$S_-$.

Now assume that  $\pi_u\,_*\,\nu=0$. Then  $\nu|_{E_u}=0$,  which we may
write as  $\nu|_{E_u}=-R_{u,v}\,_*\,\nu|_{E_u}$.
Then
\begin{eqnarray}
0&=&\int_{H_u} g(y)\,\mbox{d}\pi_u\,_*\nu=
\int_Sg\bigl(\pi_u(x)\bigr)\,\mbox{d}\nu(x)\nonumber\\
&=&\int_{E_u}g\bigl(\pi_u(x)\bigr)\,\mbox{d}\nu(x)
+\int_{S_+}g\bigl(\pi_u(x)\bigr)\,\mbox{d}\nu(x)+
\int_{S_-}g\bigl(\pi_u(x)\bigr)\,\mbox{d}\nu(x)\nonumber\\
&=&\int_{S_+}g\bigl(\pi_u(x)\bigr)\,\mbox{d}\nu(x)+
\int_{S_-}g\bigl(\pi_u(x)\bigr)\,\mbox{d}\nu(x).\label{eq:remcharly}
\end{eqnarray}
 Since $Q(u) \neq 0$,  Lemma \ref{lem:reflection} asserts  that if $x\in S_\pm$ then $\rr _{u,v}(x)\in S_\mp$
and $\pi_u (x)=\pi_u\bigl(\rr _{u,v}(x)\bigr)$, and % . Finally
% Next, from Lemma \ref{lem:reflection} we further get that if $x\in S_\pm$ then $\rr _{u,v}(x)\in S_\mp$,
% $\pi_u (x)=\pi_u\bigl(\rr _{u,v}(x)\bigr)$ and 
$\rr _{u,v}^2(x):=\rr _{u,v}\bigl(\rr _{u,v}(x)\bigr)=x$. But then
\begin{eqnarray*}
\int_{S_-}g\bigl(\pi_u(x)\bigr)\,\mbox{d}\nu(x)
&=&\int_{S_-}g\Bigl(\pi_u\bigl(\rr _{u,v}^2(x)\bigr)\Bigr)\,\mbox{d}\nu(x)\\
&=&\int_{S_+}g\Bigl(\pi_u\bigl(\rr _{u,v}(x)\bigr)\Bigr)\,\mbox{d}\rr _{u,v}\,_*\nu(x)\\
&=&\int_{S_+}g\bigl(\pi_u(x)\bigr)\,\,\mbox{d}\rr _{u,v}\,_*\nu(x).
\end{eqnarray*}
Therefore \eqref{eq:remcharly} reads as 
\begin{equation}
  \label{eq:c47}
\int_{S_+}g\bigl(\pi_u(x)\bigr)\,\mbox{d}[\nu(x)+\rr _{u,v}\,_*\nu(x)]=0.  
\end{equation}

Since  every (continuous)  function on $S_+$ can be written in the
form $g\bigl(\pi_u(x)\bigr)$ for some (continuous) function $g$ on $H_u$, 
we get $\nu_{u,s}=-R_{u,v}\,_*\,\nu_{u,s}$ on $S_+$. Replacing $S_+$ by $S_-$ in the argument shows that 
$\nu_{u,s}=-R_{u,v}\,_*\,\nu_{u,s}$ on $S_-$ as well.

(iii) $\, \Rightarrow \, $ (ii) Conversely, assume that
$\nu_{u,s}+ R_{u,v}\,_*\,\nu_{u,s} = 0$. Since $\nu _{E_u} = \rr
_{u,v} \, _* \, \nu | _{E_u}$ by symmetry, it follows that
$\nu_{u,s}|_{E_u} =0$, consequently $\pi _u \, _* \, \nu | _{E_u}
=0$. Reading  \eqref{eq:c47} and \eqref{eq:remcharly} backwards, we
find that also $\int _{S \setminus E_u} g(\pi _u(x)) \mbox{d} \nu (x)
= 0$ and thus $\pi _u \, _* \, \nu = 0$.

% *******

%  Since $Q(u) \neq 0$,  Lemma \ref{lem:reflection} asserts  that if $x\in S_\pm$ then $\rr _{u,v}(x)\in S_\mp$
% and $\pi_u (x)=\pi_u\bigl(\rr _{u,v}(x)\bigr)$. Finally
% \begin{eqnarray*}
% 0&=&\int_{H_u} g(y)\,\mbox{d}\pi_u\,_*\nu (y) =
% \int_Sg\bigl(\pi_u(x)\bigr)\,\mbox{d}\nu(x)\\
% &=&\int_{E_u}g\bigl(\pi_u(x)\bigr)\,\mbox{d}\nu(x)
% +\int_{S\setminus E_u}g\bigl(\pi_u(x)\bigr)\,\mbox{d}\nu(x)\\
% &=&\int_{S_+}g\bigl(\pi_u(x)\bigr)+g\Bigl(\pi_u\bigl(\rr _{u,v}(x)\Bigr)\bigr)\,\mbox{d}\nu(x)\\
% &=&\int_{S_+}g\bigl(\pi_u(x)\bigr)\,\mbox{d}[\nu(x)+\rr _{u,v}\,_*\nu(x)].
% \end{eqnarray*}
% Thus $\nu_{u,s}=-R_{u,v}\,_*\,\nu_{u,s}$ both on $E_u$ and on $S\setminus E_u$.

\smallskip

(iii) $\, \Leftrightarrow \, $ (iv) Finally % to establish the equivalence with (\ref{fund4}),
we  note that
$$
\widehat{\mu}(\xi+su)=\int_S e^{i\scal{x,\xi+su}}\,\mbox{d}\mu(x)
=\int_S e^{i\scal{x,\xi}}\,\mbox{d}\nu(x) = \hat{\nu }(\xi ), 
$$
whereas 
\begin{eqnarray*}
\int_S e^{i\scal{x,\xi}}\,\mbox{d}R_{u,v}\,_*\,\nu(x)&=&\int_S e^{i\scal{R_{u,v}x,\xi}}\,\mbox{d}\nu(x)\\
&=&e^{-2i\frac{\scal{u,v}\scal{u,\xi}}{Q(u)}}\int_S e^{i\scal{R_{u}x,\xi}}\,\mbox{d}\nu(x)\\
&=&e^{-2i\frac{\scal{u,v}\scal{u,\xi}}{Q(u)}}\int_S e^{i\scal{x,R_{u}^*\xi+su}}\,\mbox{d}\mu(x)\\
&=&e^{-2i\frac{\scal{u,v}\scal{u,\xi}}{Q(u)}}\widehat{\mu}(R_{u}^*\xi+su).
\end{eqnarray*}
It follows that (\ref{fund3}) and (\ref{fund4}) are equivalent.
\end{proof}

\begin{corollary} \label{cor:c}
Let $B$ be a bilinear form on $\R^d\times\R^d$ and $Q$
be the associated quadratic form $Q(x)=B(x,x)$.
Let $u,v\in\R^d$ with $|u|=1$,  $s>0$, $\rho \in \R $,   % $\eps\in\{0,1\}$,
and  let $S=\{x\in\R^d\,:Q(x)+2\scal{x,v}=\rho \}$.
\begin{enumerate}
\renewcommand{\theenumi}{\roman{enumi}}
\item $(S,H_{u,s})$ is a Heisenberg uniqueness pair if and only if $\pi_u$ is one-to-one on $S$.

\item If $Q(u)\not=0$, $(S,H_{u,s})$ is \emph{not} a Heisenberg uniqueness pair.

\item Assume that  $Q(u)=0$ and   $\mu\in\mm(S)$. Then
  $\widehat{\mu}$ vanishes on $ H_{u,s}$, 
if and only if $\mu$ is supported on the set
$$
E_{u}=\{x\in S\,: B(x,u)=-\scal{v,u}\}
$$
and  $\dst\int_{E_{u}}e^{is\scal{x,u}}\,\mbox{d}\mu(x)=0$. In particular,
if $E_{u}\cap S$ has measure zero with respect to  the surface measure
on $S$, then $(S,H_{u,s})$  \emph{is} an $\aa\cc(S)$-Heisenberg
uniqueness pair.
\end{enumerate}
\end{corollary}

\begin{proof} 
(i) Assume $\pi_u$ is one-to-one on $S$  and $\mu\in\mm(S)$.
Lemma \ref{lem:fund} shows that if  $\widehat{\mu}=0$ on
$H_{u,s}$ then $\pi_u\,_*\,\nu_{u,s}=0$ with $\mbox{d}\nu_{u,s}(x)=e^{is\scal{x,u}}\,\mbox{d}\mu(x)$.
Since $\pi _u$ is one-to-one, it follows that $\nu_{u,s}=0$,  thus $\mu=0$ and $(S,H_{u,s})$ is a Heisenberg uniqueness pair.

Conversely, if $\pi_u$ is not one-to-one on $S$, there exist points
$x,y \in S$, $x\neq y$  with $\pi_u(x)=\pi_u(y)$.
Let $\mu=e^{-is\scal{x,u}}\delta_x-e^{-is\scal{y,u}}\delta_y \neq 0$,  so that $\nu_{u,s}=\delta_x-\delta_y$.
Therefore $\pi_u\,_*\,\nu_{u,s}=0$ and Lemma \ref{lem:fund} shows that $\widehat{\mu}=0$ on
$H_{u,s}$.

(ii) If  $Q(u)\not=0$ and $x\in S$ satisfies $B(x,u)+\scal{v,u}\not=0$, then Lemma \ref{lem:reflection}
shows that $y=R_{u,v}x\in S\setminus\{x\}$ and $\pi_u(x)=\pi_u(y)$. The previous argument shows that
$(S,H_{u,s})$ is not a Heisenberg uniqueness pair. The proof easily adapts to show that 
$(S,H_{u,s})$ is not an $\aa\cc (S)$-Heisenberg uniqueness pair.

(iii) If $Q(u)=0$, then $\pi_u\,:S\to H_u$ is one-to-one on $S\setminus
E_{u}$ by  Lemma \ref{lem:reflection}. Therefore if
$\pi_u\,_*\,\nu_{u,s}=0$ then $\nu_{u,s}=0$ on $S\setminus E_{u}$ and thus $\mu=0$ on $S\setminus E_{u}$.
\end{proof}

As  a single  hyperplane  leads to a Heisenberg uniqueness pair only in
exceptional cases,
we are  led to ask when  $(S,H_{u_1,s_1}\cup H_{u_2,s_2}\cup\cdots\cup
H_{u_N,s_N})$ for  $N\geq2$ can be a
Heisenberg uniqueness pair. In this case, the measure $\mu$ must
satisfy multiple symmetries,  and those symmetries may be
incompatible. Our aim is to show that this is indeed  the case. To do so, we will first investigate
measures on $S$ that have Fourier transform vanishing on two intersecting hyperplanes.

\subsection{Proof of Theorem \ref{th:dimd}. Measures on quadratic
  surfaces with characteristic function vanishing on two hyperplanes}\  

We now  prove Theorem \ref{th:dimd}. Throughout this section, $Q$, $B$ and $S$ will be as above.
We will  prove that, unless $(u_1,u_2)$  is in an  exceptional set
$\ee$  of measure $0$,
a measure $\mu\in\mm(S)$ such that $\widehat{\mu}=0$ on $H_{u_1}\cup H_{u_2}$ is necessarily $\mu=0$.

We start with an appropriate parametrization of $\S^{d-1}\times \S^{d-1}$. 
For given   $u_1\in \S^{d-1}$,  every $u_2\in\S^{d-1}$ 
can be written as  $u_2=\cos\theta \, u_1+\sin\theta \, v_2$ with  unique
$v_2\in\S^{d-1}\cap\{u_1\}^\perp$ (a $d-2$ dimensional sphere) 
and $\theta\in[0,\pi]$. Note that $\vect(u_1,u_2)=\vect(u_1,v_2)$.

Then  the Lebesgue measure on $\S^{d-1}\times \S^{d-1}$ decomposes as
\begin{multline*}
\int_{\S^{d-1}\times \S^{d-1}} \ffi(u_1,u_2)\,\mbox{d}\sigma_d (u_1)\,\mbox{d}\sigma_d (u_2)\\
=\int_{\S^{d-1}}\int_{\S^{d-1}\cap\{u_1\}^\perp}\int_0^{\pi}
\ffi(u_1,\cos\theta \, u_1+\sin\theta \, v_2)\,\sin ^{d-2} \theta \,
\mbox{d} \theta\,\mbox{d}\sigma_{d-1} (v_2)\,\mbox{d}\sigma_d (u_1)\, .
\end{multline*}
Indeed, this is just a way to write $\mbox{d} \sigma _d(u_2)$ in
spherical coordinates.

To see that the exceptional set  $\ee$ has measure zero, we will show
that,  for $u_1,v_2$ fixed, the set of angles $\theta$
for which $(u_1,\cos\theta u_1+\sin\theta v_2)\in\ee$ is countable.

Note also that $S_Q:=\{u\in\S^{d-1}\,:Q(u)=0\}$ is a lower dimensional sub-manifold of $\S^{d-1}$
and has therefore measure $0$ in $\S ^{d-1}$, and  so does  $S_Q\times S_Q$ in $\S^{d-1}\times \S^{d-1}$.

From now on,  we fix $u_1\in\S^{d-1}$, $v_2\in\S^{d-1}\cap\{u_1\}^\perp$  and write 
$u_2= u_2(\theta)= \cos\theta \, u_1+\sin\theta \, v_2$. % To simplify notation, we write $u_2=u_2(\theta)$ when the
% dependence on $\theta$ has no influence.
We further assume that $u_1,u_2\notin S_Q$ and $u_1 \neq u_2$. Our aim is to show that
there is an at most countable set of $\theta$'s 
for which $(S,H_{u_1}\cup H_{u_2})$ is not an \HUP.

\textbf{Step 1.} Let $\mu\in\mm(S)$ be such that $\widehat{\mu}=0$ on $H_{u_1}\cup H_{u_2}$.
We will write $\rr_j=\rr_{u_j}^Q$ for the involutions on $S$ defined
by $u_1$ and $u_2$. Then the equivalence (i) $\, \Leftrightarrow \, $
(iii) of Lemma~\ref{lem:fund} implies that 
$$
\rr_j\,{}_*\mu=-\mu,\ j=1,2\quad\mbox{thus}\quad
(\rr_2\circ\rr_1)\,{}_*\mu=\mu.
$$

Let $\pi$ be the orthogonal projection on the subspace $\vect(u_1,u_2)$ generated by $u_1,u_2$ and $\pi^\perp=I-\pi$
be the orthogonal projection on the orthogonal complement of
$\vect(u_1,u_2)$. Let $S_0$ be a measurable  subset of $S$ such that
$\pi ^\perp $ is one-to-one
from $S_0$ to $\vect(u_1,u_2)^\perp$. Consequently, we may write every
$y\in S$ as $ y=  x + u$ for a unique $ x \in S_0$ and $u
\in \mathrm{span}\, (u_1,u_2)$. 

We now make to crucial observations: (i)  The intersection of $S$ with
the affine plane $ x + \mathrm{span}\, (u_1,u_2)$ is a conic
section (an ellipse, parabola, hyperbola, or two lines), and (ii)
every such cross section $S\cap ( x + \mathrm{span}\, (u_1,u_2))$ is
invariant under  the involutions $\rr _1$ and $\rr _2$. This follows from
Lemma~\ref{lem:reflection}.  

We may therefore  use the analysis of Theorem~\ref{th:dim2} for each
section and each pair of lines $x+\R u_1$ and $x+\R u_2$. To apply
this analysis, we further need an appropriate restriction of the given
measure $\mu $ to the cross section $S\cap ( x + \mathrm{span}\,
(u_1,u_2))$. Technically, this is done by desintegration. 

Let $\nu=\pi^\perp\,_* \, \mu$ and let us write the disintegration
theorem (e.g.,~\cite{DM}) in the form
\begin{equation}
  \label{eq:desint4}
\int_S
\ffi(x)\,\mbox{d}\mu(x)=\int_{S_0}\int_{(\pi^\perp)^{-1}(x)}\ffi(y)\,\mbox{d}\mu_x(y)\,\mbox{d}\nu(x)\qquad\ffi\in\cc_c(S).   
\end{equation}

Recall that the measure $\mu_x$ is  uniquely determined 
by this formula  for $\nu$-almost every $x$ and that  $\mu _x$  is supported on $S\cap\bigl(x+\vect(u_1,u_2)\bigr)$.

Since $\pi ^\perp u_j = 0$ for $j=1,2$, \eqref{eq:c2} implies that
$\pi^\perp\rr_j=\pi^\perp$. Therefore,  for $\ffi\in\cc_c(S)$, we have 
\begin{eqnarray}
\int_S \ffi(x)\,\mbox{d}\rr_j\,{}_*\mu(x)&=&\int_S
\ffi\bigl(\rr_j(x)\bigr)\,\mbox{d}\mu(x) \notag \\
&=&\int_{S_0}\int_{(\pi^\perp)^{-1}(x)}\ffi\bigl(\rr_j(y)\bigr)\,\mbox{d}\mu_x(y)\,\mbox{d}\nu(x)
\notag \\
&=&\int_{S_0}\int_{(\pi^\perp)^{-1}(x)}\ffi(y)\,\mbox{d}\rr_j\,{}_*\mu_x(y)\,\mbox{d}\nu(x). \label{eq:e9}
\end{eqnarray}

It follows that $\rr_j\,{}_*\mu=-\mu$ if and only if
$\rr_j\,{}_*\mu_x=-\mu_x$ for $\nu $-almost every $x$.
The equivalence (i) $ \, \Leftrightarrow \, $ (iii) of
Lemma~\ref{lem:fund} implies that   $\widehat{\mu_x}=0$ on $\R u_1\cup\R u_2$.\footnote{The Fourier transform
of $\mu_x$ is here taken in $x+\vect(u_1,u_2)$. That is 
$$\widehat{\mu_x}(\xi u_1+\eta u_2)=\int_{\R^2}e^{i(s\xi+t\eta)}
\,\mbox{d}\mu_x(x+su_1+tu_2).$$}

We  may summarize this  reduction to the two-dimensional
case as follows. 

\begin{lemma} \label{lem:reduc}
  Let $\mu \in \mm (S)$ and $u_1,u_2 \in \S ^{d-1}$. Then
  $\widehat{\mu } $ vanishes on $H_{u_1} \cup H_{u_2}$, if and only if
  $ \widehat{\mu _x}$ vanishes on $\R u_1 \cup \R u_2$ for $\nu
  $-almost all $x\in S_0$. 
\end{lemma}

Next,  for $x\in S_0$ the measure   $\mu_x$ is  supported on the intersection of the quadratic surface $S$
with the plane $x+\vect(u_1,u_2)= x+\vect(u_1,v_2)$. This intersection
is either the union of two lines, a parabola, a hyperbola or an
ellipse. 
For each of these, we know from Theorem~\ref{th:dim2} when the
condition $\widehat{\mu _x}|_{\R u_1 \cup \R u_2}= 0$ implies that $\mu _x = 0$.

% symmetry $\rr_j\,{}_*\mu_x=-\mu_x$, $j=1,2$ implies $\mu_x=0$.

\textbf{Step 2.} We  now apply Theorem \ref{th:dim2} 
to $S\cap (x+\vect(u_1,u_2))$ and $\mu _x$. 
We first  determine the intersection $S_x = S\cap \bigl(x+\vect(u_1,u_2)\bigr)
=S\cap \bigl(x+\vect(u_1,v_2)\bigr)$ precisely.

% Let us simplify notation and write $u=u_1$ and $u'$ a vector in the plane $\vect(u_1,u_2)$ not collinear to $u_1$. We
% will choose either  $u'=u_2$ or $u'=u_1^\perp$ a vector orthogonal to $u_1$.
As $x\in S$, a point $x+su_1 +t v_2$ belongs to $S$, if and only if 
$$
Q(x+su_1+tv_2)+2\scal{x+su_1+tv_2,v}=Q(x)+2\scal{x,v} = \rho \, ,
$$
or equivalently, 
\begin{equation}
\label{eq:conique}
s^2Q(u_1)+2stB(u_1,v_2)+t^2Q(v_2)+2sb(x)+2tb'(x)=0.
\end{equation}
with $b(x)=B(x,u_1)+\scal{u_1,v}$ and $b'(x)=B(x,v_2)+\scal{v_2,v}$.

Let $\Sigma $ be the set of points $(s,t) \in \R^2$ satisfying
\eqref{eq:conique}. Then $\Sigma $ is either a hyperbola, an ellipse,
or a parabola in $\R ^2$ (where a degenerate hyperbola is a set of two
intersecting lines and a degenerate ellipse is a point or the empty
set). The classification into these three types depends on the value
of the discriminant of the quadratic form $\tilde Q (s,t) =
s^2Q(u_1)+2stB(u_1,v_2)+t^2Q(v_2)$, namely $\Delta = \det 
\Big( \begin{smallmatrix}
  Q(u_1) & B(u_1,v_2) \\ B(u_1,v_2) & Q(v_2) 
\end{smallmatrix}\Big) = Q(u_1) Q(v_2) - B(u_1,v_2)^2$. The set 
$$
S_x = S\cap \bigl(x+\vect(u_1,v_2)\bigr)=  \{ x + su_1 + tv_2 : (s,t) \in \Sigma \} \subseteq \R ^d
$$ 
is obtained from $\Sigma $ by a linear transformation and a shift
by $x$ and thus  represents the same conic section as $\Sigma $. 
Specifically, if $Q(u_1)Q(v_2)-B(u_1,v_2)^2<0$, then, for all $x\in
S_0$, $ S\cap \bigl(x+\vect(u_1,v_2)\bigr)$ is a 
 hyperbola or two intersecting lines (in the degenerate case). If
 $Q(u_1)Q(v_2)-B(u_1,v_2)^2>0$, then $ S\cap \bigl(x+\vect(u_1,v_2)\bigr)$ is an 
 ellipse or a point (in the degenerate case). If
 $Q(u_1)Q(v_2)-B(u_1,v_2)^2=0$, then $ S\cap
 \bigl(x+\vect(u_1,v_2)\bigr)$ is a parabola  or a set of two parallel
 lines or a single line (in the degenerate case).

Theorem~\ref{th:dim2} states the existence of  a set $\ee_x=\ee(\tilde
Q)\subseteq \vect(u_1,v_2)$ of pairs of directions such that if   $(w_1,w_2)\notin\ee_x$ 
then $\mu_x=0$. 
At this point it is crucial that $\ee _x$ depends only on $\tilde Q$,
but not on $x\in S_0$.

Consequently, if $w_1,w_2 \not \in \ee(\tilde
Q)$, then $\mu _x = 0$ for \emph{all} $x\in S_0$. By
\eqref{eq:desint4} 
 this implies that $\mu = 0$.

\textbf{Step 3.} We finally calculate the measure of $\ee $ in $\S
^{d-1} \times \S ^{d-1}$.  Fix $u_1$ and $v_2 \in u_1^\perp $,  and
take  $w_1 = u_1$ and $w_2=\cos \theta \, u_1 + \sin \theta \, v_2$.
Then the set of $\theta$, such that  % $(w_1,w_2) \in \ee (\tilde Q)$ is countable
% for which
$\bigl(w_1,w_2\bigr)\in\ee_x$ is contained in a fixed countable set
independent of $x$. Consequently the measure of the exceptional set
$\ee $ is 
\begin{multline}
\int_{\S^{d-1}\times \S^{d-1}} \chi _{\ee } (u_1,u_2)\,\mbox{d}\sigma_d (u_1)\,\mbox{d}\sigma_d (u_2)\\
=\int_{\S^{d-1}}\int_{\S^{d-1}\cap\{u_1\}^\perp}\int_0^{\pi}
\chi _{\ee } (u_1,\cos\theta \, u_1+\sin\theta \, v_2)\,\sin ^{d-2} \theta \,
\mbox{d} \theta\,\mbox{d}\sigma_{d-1} (v_2)\,\mbox{d}\sigma_d (u_1)
 =0
\, .
\end{multline}

This completes the proof ot Theorem \ref{th:dimd}.

\section{The cone and the hyperboloid}

In this section we investigate special arrangements of hyperplanes for
$S$ being a  hyperboloid or a  cone. 

Let $p$ be an integer, $1\leq p\leq d-1$ and $q=d-p$. Let $Q$ be the quadratic form
$$
Q(x)=x_1^2+\cdots+x_p^2-x_{p+1}^2-\cdots-x_{p+q}^2
$$
with the associated bilinear form  $B$  and set $v=0$. Let $\rho\in\R$
and $S_\rho=\{x\in\R^d\,:\ Q(x)=\rho\}$. Then 
$S_0$ is a cone and  $S_\rho, \rho \neq 0,$ is a hyperboloid with
one or two connected components.  
For instance, in dimension 3, if $p=1$ and $q=2$ then $S_1$ has 2 sheets while $S_{-1}$ has only one.

Our aim  is to complement the statement of Theorem \ref{th:dimd} which treats the case of
pairs of intersecting hyperplanes $H_{u_1}\cup H_{u_2}$ with non-isotropic normals, $Q(u_1),Q(u_2)\not=0$.
We will treat two cases. On one hand, we will consider the case of parallel hyperplanes and show that, generically,
$(S,H_{u,s_1}\cup H_{u,s_2}\cup H_{u,s_3})$ is a \HUP. On the other hand, we will also show that it is possible
to construct Heisenberg uniqueness pairs of the form $(S,H_{u_1}\cup\cdots H_{u_k})$ where all normals $u_1,\ldots,u_k$
are isotropic vectors.

\smallskip

In this section, we will write $\R^d=\R^p\times\R^q$ in the obvious sense:  if
$x\in\R^d$, then $x'$ is  its projection on $\R^p$
and $x''$ its projection on $\R^q$ so that $x=(x',x'')$. In
particular, $\langle x, u\rangle = \langle x', u'\rangle +   \langle
x'', u''\rangle $  and $B(x,u) =  \langle x', u'\rangle -   \langle
x'', u''\rangle $. 

Note also that, if $|u|=1$ and $Q(u)\not=0$, then
\begin{eqnarray*}
\rr _u x&=&\scal{x,u}u+\pi_{u^\perp}x-2\frac{B(x,u)}{Q(u)}u\\
&=&\frac{\bigl(Q(u)-2\bigr)\scal{x',u'}+\bigl(Q(u)+2\bigr)\scal{x'',u''}}{Q(u)}u+\pi_{u^\perp}x.
\end{eqnarray*}
Set 
$$
\tilde
u=\frac{\Big(\bigl(Q(u)-2\bigr)u',\bigl(Q(u)+2\bigr)u''\Big)}{Q(u)}
\, ,
$$
then $\rr _u ^* u = \tilde u = \langle \tilde u ,u \rangle u + \pi
_{u^\perp } \tilde u$. Since the normalization  $|u'|^2+|u''|^2=1$
implies that  $|u'|^2=\dst\frac{1+Q(u)}{2}$
and $|u''|^2=\dst\frac{1-Q(u)}{2}$, we thus  obtain $\langle \tilde u ,u
\rangle = -1$. 
An easy computation now  shows that
\begin{eqnarray*}
\rr _u^*\xi&=& \rr _u ^* \big( \langle \xi ,u\rangle u + \pi _{u^\perp
  } \xi \big) \\
&=& \scal{\xi,u} \rr _u^* u + \pi_{u^\perp}\xi\\
&=&-\scal{\xi,u}u
+\scal{\xi,u}\pi_{u^\perp}\tilde u+\pi_{u^\perp}\xi \, .
\end{eqnarray*}

% \begin{eqnarray*}
% \rr _u^*\xi&=&\scal{\xi,u}\frac{\Bigl(\bigl(Q(u)-2\bigr)u',\bigl(Q(u)+2\bigr)u''\Bigr)}{Q(u)}+\pi_{u^\perp}\xi\\
% &=&-\scal{\xi,u}u
% +\scal{\xi,u}\pi_{u^\perp}\tilde u+\pi_{u^\perp}\xi
% \end{eqnarray*}

Recall from Lemma~\ref{lem:fund} that, when $Q(u)\not=0$ and $\mu\in\mm(S)$, then $\widehat{\mu}=0$ on $H_{u,s}$ if and only if
\begin{equation}
\label{eq:cone}
\widehat{\mu}(\xi+su)=-\widehat{\mu}(\rr_u^*\xi+su)\qquad\mbox{for all }\xi\in\R^d.
\end{equation}

\subsection{Parallel hyperplanes with non-isotropic normal}\ 

We will now prove the following statement. 

\begin{proposition} Let $1\leq p , q\leq d-1$ be two integers with $
  p+q=d$,   $s_1,s_2, s_3 \in \R $,  and let 
$$
Q(x)=x_1^2+\cdots+x_p^2-x_{p+1}^2-\cdots-x_{p+q}^2.
$$
Let $u\in\R^d$ be such that $Q(u)\not=0$, $\rho\in\R$,  and set $S_\rho=\{x\in\R^d\,:Q(x)=\rho\}$.
If $s_1,s_2,s_3\in \R$ are linearly independent over $\Q$,
then $(S_\rho,H_{u,s_1}\cup H_{u,s_2}\cup H_{u,s_3})$ is a Heisenberg
uniqueness pair. 
\end{proposition}

\begin{proof}
Translating the hyperplanes in the direction $u$, it is enough to consider the case $s_1=0$.
We thus want to prove that
$(S_\rho,H_{u,0}\cup H_{u,s}\cup H_{u,t})$ is a Heisenberg uniqueness pair when
$s,t$ are $\Q$-independent.

From \eqref{eq:cone}, we obtain that $\widehat{\mu}$ satisfies the following symmetries
$$
\left\{\begin{matrix}
\widehat{\mu}(\xi)&=&-\widehat{\mu}(\rr_u^*\xi)\\
\widehat{\mu}(\xi+su)&=&-\widehat{\mu}(\rr_u^* \xi+su)\\
\widehat{\mu}(\xi+tu)&=&-\widehat{\mu}(\rr_u^* \xi+tu)
\end{matrix}\right.
$$
Now for arbitrary  $\xi\in H_{u,0}=u^\perp$ write $\xi+s_0u=\xi+(s_0-s)u+su$ and note that
$\rr_u^*\bigr(\xi+(s_0-s)u\bigl)=(s-s_0)u+(s_0-s)\pi_{u^\perp}\tilde
u+\xi$, since $\pi _{u^\perp } \xi = \xi $. It follows that, on the one hand, 
\begin{equation}
\label{eq:coneII1}
\widehat{\mu}(\xi+s_0u)=-\widehat{\mu}\bigr(\rr_u^*(\xi+s_0 u )\bigr)
=-\widehat{\mu}(-s_0u+s_0\pi_{u^\perp}\tilde u+\xi) \, ,
\end{equation}
and on the other hand
\begin{align}
\widehat{\mu}(\xi+s_0u) = \widehat{\mu } (\xi + (s_0-s)u + su) &=
- \widehat{\mu}\bigr(\rr ^*_u ( \xi + (s_0-s)u ) + su\bigr) \notag \\
&= -\widehat{\mu}\bigr(\xi+(s_0-s)\pi_{u^\perp}\tilde
u+(2s-s_0)u\bigl)\, .  \label{eq:coneII2}
\end{align}

% \begin{equation}
% \label{eq:coneII2}
% \widehat{\mu}(\xi+s_0u) = \widehat{\mu } (\xi + (s_0-s)u + su) =
% - \widehat{\mu}\bigr(\rr ^*_u ( \xi + (s_0-s)u ) + su\bigr) = 
% -\widehat{\mu}\bigr(\xi+(s_0-s)\pi_{u^\perp}\tilde
% u+(2s-s_0)u\bigl)\, . 
% \end{equation}
Now assume that $\widehat{\mu } = 0$ on $H_{u,s_0}$.  Then \eqref{eq:coneII1} implies that
$\widehat{\mu}=0$ on $H_{u,-s_0}$. Applying now  \eqref{eq:coneII2}
with $-s_0$ instead of $s_0$, we obtain that
 $\widehat{\mu}=0$ on $H_{u,s_0+2s}$. By induction, it follows that $\widehat{\mu}=0$ on
$H_{u,s_0+2ks}$, for every $k\in\N$. Reversing the order in which we use \eqref{eq:coneII1}-\eqref{eq:coneII2}
we obtain  that $\widehat{\mu } =0$ on $H_{u,s_0+2ks}$ for every $k\in\Z$.

Next,  using the third symmetry we obtain  that $\widehat{\mu}=0$ on
$H_{u,s_0+2ks+2\ell t}$ for every $k, \ell\in\Z$. 
As $\widehat{\mu}=0$ on $H_{u,0}$ we get that $\widehat{\mu}=0$ on
every hyperplane 
$H_{u,2ks+2\ell t}$ for every $k,\ell\in\Z$. This set is dense in $\R^d$ and $\widehat{\mu}$ is continuous,
therefore $\widehat{\mu}=0$ and $(S,H_{u,0}\cup H_{u,s}\cup H_{u,t}) $ is a \HUP . 
\end{proof}

\subsection{Intersecting hyperplanes with isotropic normal}\ 

Again, let $1\leq p\leq q\leq d-1, p+q = d$ with $p+q=d$, 
$Q(x)=x_1^2+\cdots+x_p^2-x_{p+1}^2-\cdots-x_{p+q}^2 $, and 
$S_\rho=\{x\in\R^d\,:Q(x)=\rho\}$ for $\rho \in \R $.
We now assume that $u\in\R^d$ is isotropic, i.e.,  $Q(u)=0$. Write
$u=(u',u'')\in\R^p\times\R^q$,  then $Q(u)=0$ is equivalent to
$|u'|=|u''|$.
Without loss of generality, we can restrict our attention to $u$ with $|u'|=|u''|=1$. 

Let $\mu \in \mathcal{M}(S_\rho)$. By Corollary~\ref{cor:c}, $\hat{\mu }$ vanishes on
$H_u$, if and only if $\mu$ is supported on the set 
$E_u =\{x\in S_\rho  : B(x,u)=0\} $. We represent $E_u$ in a
different form as follows: a point $(x',x'')$ is in $ S_\rho$, if and
only if $|x'|^2 - |x''|^2 = \rho $, i.e., $x' \in (|x''|^2 + \rho
)^{1/2} \S ^{p-1}$, and $B(x,u) = 0$, if and only if
$\scal{x',u'}=\scal{x'',u''}$. 
Consequently, 
\begin{equation}
\label{eq:c7} 
E_u= \bigr\{(x',x'')\in\R^{p+q}\,: x'\in
(|x''|^2 + \rho )^{1/2} \S^{p-1}\cap\bigl(\scal{x'',u''}u'+(u')^\perp\bigr)\bigr\}.  
\end{equation}
We note right away that $E_u$ is symmetric,   and thus   $x\in E_u$ implies that $-x\in
E_u$.

This  leads to the first obversation on intersecting hyperplanes with
isotropic normals.

% We will now prove the following proposition for a finite number of
% isotropic   hyperplanes.
%are all isotropic, i.e., $Q$ vanishes on their normal vectors. 

\begin{proposition} \label{prop:iso2} \
\begin{enumerate}
\renewcommand{\theenumi}{\roman{enumi}}
\item If $Q(u)=0$ for non-zero $u$, then $(S_\rho,H_u)$ is an $\aa\cc(S)$-Heisenberg uniqueness pair.

\item Given vectors  $u_j, j=1, \dots , k$, such that
$Q(u_1)=\cdots=Q(u_k)=0$, $(S_\rho,H_{u_1}\cup\cdots\cup H_{u_k})$ is
a \HUP , if and only if $\bigcap _{j=1}^k E_{u_j} $ is either empty or
$\{ 0\}$. 

\item If $k<p$ and $Q(u_1)=\cdots=Q(u_k)=0$, then $(S_\rho,H_{u_1}\cup\cdots\cup H_{u_k})$ is not a Heisenberg uniqueness pair.
\end{enumerate}
\end{proposition}

\begin{proof}
(i) Since  $E_u$ is a $(d-2)$-dimensional submanifold of the
$(d-1)$-dimensional manifold $S_\rho$, it  has surface measure zero in
$S_\rho$.  An absolutely continuous  measure  with 
support in $E_u$ is necessarily  $0$ (Corollary~\ref{cor:c}(iii)).

(ii) Next assume that $u_1,\ldots,u_k$ are such that
$Q(u_1)=\cdots=Q(u_k)=0$. By Corollary~\ref{cor:c}, $\hat{\mu }$ vanishes on
$\bigcup _{j=1}^k H_{u_j}$, if and only if $\mu  $ is supported on 
$\bigcap _{j=1}^kE_{u_j} $. If this intersection is empty, then
obviously $\mu = 0$. If this intersection is $\{0\}$, then 
$\mu=c\delta_0$. As $\widehat{\mu}(0)=0$, we obtain $\mu=0$. Note
that if $\rho\not=0$, then $0\notin S_\rho$ and we obtain directly that
$\mu=0$. 

If $\bigcap _{j=1}^kE_{u_j} $ contains at least two points, then there
is a non-zero $x$ such that $\{x, -x\} \subset \bigcap _{j=1}^kE_{u_j} $ by
the symmetry of $E_u$. Now 
set   $\mu=\delta_x-\delta_{-x}\in\mm(S)$. Then $\mu $  is non-zero, but $\widehat{\mu}=0$
on $H_{u_1}\cup \cdots\cup H_{u_k}$, so $(S_\rho, H_{u_1}\cup\cdots\cup
H_{u_k})$ is not a \HUP .

(iii)  If $k<p$, then   there exists a non-zero  $x''\in
(u^{\prime\prime}_1)^\perp\cap\cdots(u^{\prime\prime}_k)^\perp$ of
norm $|x''|^2 > |\rho | $ and a
 $x'\in(|x''|^2+\rho )^{1/2} \S^{p-1}\cap\bigcap_{j=1}^k
(u^{\prime}_j)^\perp$. Consequently  $E_{u_1}\cap\cdots\cap
E_{u_k}  $ is not empty and contains the  non-zero point
$(x',x'')$. By item (ii) % . Choose $x, -x \in
% E_{u_1}\cap\cdots\cap E_{u_k}  $ and set  
% % so that  $x=(x',x'')\in E_{u_j}$ then by symmetry  $x,-x\in E_{u_1}\cap\cdots\cap E_{u_k}$.
% % Then
% $\mu=\delta_x-\delta_{-x}\in\mm(S)$. Then $\mu $  is non-zero, but $\widehat{\mu}=0$
% on $H_{u_1}\cup \cdots\cup H_{u_k}$, so 
$(S_\rho, H_{u_1}\cup\cdots\cup
H_{u_k})$ is not a \HUP .   
\end{proof}

Using the geometric characterization (ii) of
Proposition~\ref{prop:iso2}, one can now derive numerous examples of
\HUP s. As an example we prove the following statements.

\begin{proposition}
Let $\rho \in\R$, $1\leq p \leq q\leq d-1$ with $p+q=d$, $Q$ be a quadratic form 
$$
Q(x)=x_1^2+\cdots+x_p^2-x_{p+1}^2-\cdots-x^2_{p+q}
$$
and $S_\rho$ be the cone ($\rho=0$) or hyperboloid ($\rho\not=0$) $S_\rho=\{x\in\R^d\,:Q(x)=\rho\}$.
	
\begin{enumerate}
\renewcommand{\theenumi}{\roman{enumi}}
\item  If $\rho>0$, there exist  $u_1,\ldots,u_{p}\not=0$ with $Q(u_1)=\cdots=Q(u_{p})=0$, such that
 $(S_0,H_{u_1}\cup\cdots\cup H_{u_p})$ is a Heisenberg uniqueness pair.

\item If $\rho=0$, there exist $2p$  vectors  
$u_1,\ldots,u_{2p}\not=0$ with $Q(u_1)=\cdots=Q(u_{2p})=0$ such that
$(S_0,H_{u_1}\cup\cdots\cup H_{u_{2p}})$ is a Heisenberg uniqueness pair.

\item If $\rho<0$, there exist $p+q$  vectors 
$u_1,\ldots,u_{p+q}\not=0$ with $Q(u_1)=\cdots=Q(u_{p+q})=0$ such that
$(S_0,H_{u_1}\cup\cdots\cup H_{u_{p+q}})$ is a Heisenberg uniqueness pair.
\end{enumerate}
\end{proposition}

\begin{proof} %Without loss of generality, we can assume that $p\leq q$.
Choose an orthonormal basis
$u^\prime_1,\ldots,u^\prime_p$ of $\R^p$.
Define $u^{\prime\prime}_j=(u^\prime_j,0,\ldots,0)\in\R^q$, $u_j=(u_j^\prime,u_j^{\prime\prime})$,
 and $\tilde u_j=(u_j^\prime,-u_j^{\prime\prime})$, $j=1,\ldots,p$.

First note that, for every $a_1,\ldots,a_p$, the
intersection 
$$
\bigl(a_1u^\prime_1+(u^\prime_1)^\perp\bigr)\cap\cdots\cap
\bigl(a_pu^\prime_p+(u^\prime_p)^\perp\bigr)
$$
contains exactly  one  point, namely,  $x'=\sum_{j=1}^pa_ju_j^\prime$. 

 If $(x',x'')\in\dst\cap_{j=1,\ldots,p}E_{u_j}$ then, in view of
 \eqref{eq:c7}, 
$x'\in
\cap_{j=1}^p\bigl(\scal{x'',u_j^{\prime\prime}}u^\prime_j+(u^\prime_j)^\perp\bigr)$, 
so that 
$$
x'=\sum_{j=1}^p\scal{x'',u_j^{\prime\prime}}u_j^\prime
\quad\mbox{and}\quad
|x'|^2=\sum_{j=1}^p\scal{x'',u_j^{\prime\prime}}^2.
$$
In particular, if we write $x''=(y,\tilde y)$ with $y\in\R^p$ and $\tilde y\in \R^{q-p}$
(with the obvious abuse of notation when $q=p$), then $x'=y$. %In other words

Let us first assume that $\rho>0$ then
$|x'|\leq|x''|<(|x''|^2+\rho)^{1/2}$. This contradicts \eqref{eq:c7} and thus
 implies
that $\dst\bigcap_{j=1,\ldots,p}E_{u_j}=\emptyset$. Applying
Proposition~\ref{prop:iso2},  this proves (i). 

If $\rho\leq 0$,
$$
\bigcap_{j=1,\ldots,p}E_{u_j}=\{(x',x',\tilde x),\ x'\in\R^p,\tilde
x\in\R^{q-p}\} \, ,
$$
and similarly,
$$
\bigcap_{j=1,\ldots,p}E_{\tilde u_j}=\{(x',-x',\tilde x),\
x'\in\R^p,\tilde x\in\R^{q-p}\} \, .
$$

If $q=p$, then already $\bigcap_{j=1,\ldots,p}E_{u_j}\cap \bigcap_{j=1,\ldots,p}E_{\tilde u_j}=\{0\}$.

For  $\rho=0$ the condition $|x'| = |x''| = |(x',\tilde x)|$ implies
that  $\tilde x=0$,  and again
$\bigcap_{j=1,\ldots,p}E_{u_j}\cap \bigcap_{j=1,\ldots,p}E_{\tilde u_j}=\{0\}$.

In both cases, Proposition \ref{prop:iso2} shows that
$(S_0,H_{u_1}\cup \cdots\cup H_{u_p}\cup H_{\tilde u_1}\cup \cdots\cup H_{\tilde u_p})$
is a \HUP . Thus (ii) is proved. 

Now, if $\rho<0$ and $q>p$, then
$$
\bigcap_{j=1,\ldots,p}E_{u_j}\cap \bigcap_{j=1,\ldots,p}E_{\tilde u_j}=\{(0,\ldots,0,\tilde x),\tilde x\in\R^{q-p}\}.
$$
Complete the  orthonormal set  $u_j^{\prime\prime}$, $j=1,\ldots,p$ with vectors $u_j^{\prime\prime}$, 
$j=p+1,\ldots,q$, into an orthonormal  basis of $\R^q$
and let $u_j=(u_1^\prime,u_j^{\prime\prime})$, $j=p+1,\ldots,q$. If 
$x\in \dst\cap_{j=1,\ldots,q}E_{u_j}\cap
\cap_{j=1,\ldots,p}E_{\tilde u_j}$, then on the one hand,  $x=
(0,\ldots,0,\tilde x)$ with $\tilde x\in\R^{q-p}$. On the other hand,
since  also  $x \in E_{u_j}$ for   $j=p+1,\ldots,q$, we must have
$0\in\scal{(0,\tilde x) ,u_j^{\prime \prime }}u_1^\prime +(u_1 ^\prime
)^\perp$. 
Thus $\scal{(0,\tilde x) ,u_j ^{\prime \prime }}=0$ and finally $\tilde
x=0$. Again, from Proposition~\ref{prop:iso2} we get that 
$(S_0,H_{u_1}\cup \cdots\cup H_{u_q}\cup H_{\tilde u_1}\cup \cdots\cup H_{\tilde u_p})$
is a \HUP.
\end{proof}

\section{The sphere}

In this section $Q(x)=|x|^2$, $B(x,y)=\scal{x,y}$, $v=0$ and $S=\S^{d-1}$ is the unit sphere of $\R^d$.
If $|u|=1$, $\rr _u$ is the reflection with respect to the hyperplane
normal to $u$ and $\rr _u^*=\rr _u$. 

Let us first consider the case of parallel hyperplanes.

\begin{proposition}
Let $u\in\S^{d-1}$ and $s_1\not=s_2\in\R$. Then
\begin{enumerate}
\renewcommand{\theenumi}{\roman{enumi}}
\item $(\S^{d-1},H_{u,s_1}\cup H_{u,s_2})$ is an $\aa\cc (\S ^{d-1})$-Heisenberg uniqueness pair.

\item $(\S^{d-1},H_{u,s_1}\cup H_{u,s_2})$ is a Heisenberg uniqueness pair if and only if $|s_1-s_2|>\dst\tfrac{\pi }{2}$.

\item If $s_1,s_2,s_3$ are linearly independent over $\Q$, then
  $(\S^{d-1},H_{u,s_1}\cup H_{u,s_2} \cup H_{u,s_3})$ is a Heisenberg uniqueness pair.
\end{enumerate}
\end{proposition}

The proof of (i)  is essentially the same as in the case $d=2$ by N. Lev \cite{Le}.

\begin{proof}
After a  translation, it is enough to prove that 
$(\S^{d-1},H_{u,-s}\cup H_{u,s})$ is a Heisenberg uniqueness pair,
where $s= |s_1-s_2|/ 2 $. %$s=\dst\frac{|s_1-s_2|}{2}$.
After applying a suitable  rotation, we may further assume that $u=(1,0,\ldots,0)$.
We will use the following notation: we write $x=(x_1,\bar x)$  for  $x\in\R^d$.

Let $\mu\in\mm(\S^{d-1})$.
According to Lemma \ref{lem:fund}, 
$\widehat{\mu}=0$ on $H_{u,-s}\cup H_{u,s}$ is equivalent to
$\rr_u\,_*\bigl(e^{\pm is x_1}\,\mbox{d}\mu(x_1,\bar x)\bigr)=-e^{\pm is x_1}\,\mbox{d}\mu(x_1,\bar x)$.
But $\rr_u\,_*\bigl(e^{\pm is x_1}\,\mbox{d}\mu(x_1,\bar x)\bigr)=e^{\mp is x_1}\,\mbox{d}\mu(-x_1,\bar x)$ thus
\begin{equation}
\label{eq:system}
\left\{\begin{matrix}
e^{+is x_1}\,\mbox{d}\mu(x_1,\bar x)&+&e^{- is x_1}\,\mbox{d}\mu(-x_1,\bar x)&=&0\\
e^{-is x_1}\,\mbox{d}\mu(x_1,\bar x)&+&e^{+ is x_1}\,\mbox{d}\mu(-x_1,\bar x)&=&0\\
\end{matrix}\right..
\end{equation}
Writing each equation in the form $e^{\pm 2isx_1} d\mu (x_1, \bar x) =
- d\mu (-x_1, \bar x)$ and subtracting these equations, we obtain that 
 $\sin(2 sx_1)\mbox{d}\mu(x_1,\bar x)=0$. Consequently  $\mu$ is

 supported on $\S^{d-1}\cap \big(\tfrac{\pi }{2s} \Z \times \R ^{d-1}\big)$.
Further, \eqref{eq:system} shows that $\mu=0$ on $\S^{d-1}\cap ( \{0\}\times \R ^{d-1})$ so that
$\mu$ is actually supported on $\S^{d-1}\cap \left(\tfrac{\pi }{2s} \Z^* \times \R ^{d-1}\right)$.

Note that $\mu$ is supported on a set which has surface measure zero on $\S ^{d-1}$. Thus, if
$\mu $ is absolutely continuous with respect to surface measure on
$\S ^{d-1}$, then $\mu = 0$ and $(\S^{d-1},H_{u,-s}\cup H_{u,s})$ is
an $\aa\cc (\S^{d-1})$-Heisenberg uniqueness pair.

In the general case,  $\mu=\sum_{k\in\Z^*}\delta_{\tfrac{\pi
  }{2s}k}\otimes\mu_k$. Since $\mathrm{supp}\, \mu \subseteq \S
^{d-1}$, this representation of $\mu $ implies  that $\mu_k=0$ if
$\dst\tfrac{\pi }{2s}|k|>1$  % {\it i.e.} $|k|>\tfrac{2}{\pi }s$,
and that 
$\mu_k$ is supported on $\dst\left(1-\tfrac{\pi^2 }{4s^2}k^2\right)^{1/2}\S^{d-2}$ if $|k|\leq \tfrac{2}{\pi }s$.
Moreover, the symmetry  \eqref{eq:system} implies that  $e^{i\tfrac{\pi }{2}k}\mu_k+e^{-i\tfrac{\pi }{2}k}\mu_{-k}=0$,
that is $\mu_{-k}=(-1)^{k-1}\mu_k$. In particular, if
$s>\dst\tfrac{\pi }{2}$, then  $\mu=\delta_0\otimes\mu_0$ and thus
$\mu =0$. %  This shows that, when $s>\dst\tfrac{\pi }{2}$, $\mu=0$.

On the other hand, if $s\leq\dst\tfrac{\pi }{2}$, we take the
$\mu_k$'s to be arbitrary  measures supported on 
$\dst\left(1-\tfrac{\pi^2 }{4s^2}k^2\right)^{1/2}\S^{d-2}$ for  $0<k\leq \tfrac{2}{\pi }s$ and then define
$\mu_{-k}=(-1)^{k-1}\mu_k$ for those $k$'s. The measure $\mu=\sum_{0<|k|\leq\tfrac{2}{\pi }s}\delta_{\tfrac{\pi }{2s}k}\otimes\mu_k$
satisfies $\widehat{\mu}=0$ on $H_{u,-s}\cup H_{u,s}$. Thus $(\S^{d-1}, H_{u,-s}\cup H_{u,s})$ is not a \HUP . 

Note that in this case,
$\widehat{\mu}(\xi_1,\bar\xi)=\sum_{0<|k|\leq\tfrac{2}{\pi
  }s}e^{-\tfrac{i \pi }{2s}k\xi_1}\widehat{\mu_k}(\bar\xi)$ 
is $4\pi s$ periodic in the first variable. Consequently, after a
translation,  if $\mu\in\mm(\S^{d-1})$ and  
$\widehat{\mu}= 0$ on   $H_{0,s_1}\cup H_{0,s_2}$,  then
$\widehat{\mu}$ is $2(s_2-s_1)\pi$ periodic in the first
variable $\xi _1$. Thus,  
if $\widehat{\mu}=0$ on $H_{0,s_3}$ as well, then $\widehat{\mu}$ is
also $2(s_3-s_1)\pi$ periodic in $\xi _1$. 
In particular, for every $k,\ell\in\Z$ we have
$\widehat{\mu}(s_1+2k(s_2-s_1)\pi+2\ell
(s_3-s_1)\pi,\bar\xi)=\widehat{\mu}(s_1,\bar\xi)=0$. 
If $s_1,s_2,s_3$ are linearly independent over $\Q$ then the set
$\{s_1+2k(s_2-s_1)\pi+2\ell (s_3-s_1)\pi : k,\ell \in \Z \}$ is dense
in $\R$ 
thus $\widehat{\mu}=0$ on a dense set in $\R^d$. As $\widehat{\mu}$ is
continuous, $\widehat{\mu}=0$ on $\R^d$ and thus $\mu=0$. 
\end{proof}

Let us now turn to $N$  hyperplanes $H_1,\ldots,H_N$ that intersect in a common point. Without loss of generality,
we will assume that $H_1,\ldots,H_N$ intersect in $0$ so that
$H_k=H_{u_k,0}$ for some $u_k \in \R ^d, |u_k|=1$. 

\begin{theorem} \label{th:cox}
Let $H_1,\ldots,H_N$ be hyperplanes in $\R^d$ and $\rr _1,\ldots,\rr _N$ be the corresponding orthogonal reflections.
Then $(S^{d-1},H_1\cup \cdots\cup H_N)$ is a Heisenberg uniqueness pair if and only if
the Coxeter group generated by $\rr _1,\ldots,\rr _N$ is infinite.
\end{theorem}

\begin{proof}
  Let $\rr _k$ be the reflection
with respect to $H_k$.    If  $\mu \in M (\S^{d-1}) $ and
$\widehat{\mu } $ vanishes on $H_k$, 
then  by Lemma \ref{lem:fund} we have $\rr _k \, _*\,  \mu = - \mu $. %it is invariant under $R_k$.
Consequently, its Fourier transform satisfies $|\hat{\mu }(R_k \xi )|
= |\hat{\mu }(\xi )|$ for all $\xi \in \R ^d$. %  is invariant under $R_k^*=R_k$. 
% Equivalently, its Fourier transform is invariant under $R_k^*=R_k$.

Let $G$ be the group generated by the reflections
$\{\rr _1,\ldots,\rr _k\}$. This is a {\em Coxeter} group. The modulus
$|\hat{\mu }|$ is  therefore invariant under the  Coxeter group
$G=\langle \rr _1,\ldots,\rr _k\rangle $. 
% generated by the reflections $R_1,\ldots,R_k$. 
Let  $u,u' \in \S ^{d-1}$ and $x\in H_u$. Since $\langle \rr _{u'}x,
\rr _{u'}u\rangle = \langle x,u\rangle$, we have $\rr _{u'} H_u =
H_{\rr _{u'}u}$. Consequently, if $g = \rr _{j_1}\rr _{j_2} \dots \rr _{j_n} \in
G$ and $\hat{\mu } = 0$ on $H_u$, then $\hat{\mu }$ vanishes on
$H_{gu}$. 
% Note that if $R,R'\in G$ are two reflections with respect to hyperplanes $H,H'$
% then $\widehat{\mu}=0$ on $H,H'$. Indeed, $\widehat{\mu}=0$ on $H_1,\ldots,H_N$, $|\widehat{\mu}|$ is invariant under
% $R_1,\ldots,R_N$ thus $\widehat{\mu}=0$ on every hyperplane obtained by a finite number of reflexions $R_1,\ldots,R_N$,
% in particular on $H$ and $H'$. 
Furthermore note that the composition $\rr _{u'}\rr _u$ of two reflections   is a
rotation in the plane spanned by $u$ and $u'$  with angle
twice the angle between $u$ and $u'$. 

We distinguish  two cases: 

(i) Either $G$ is infinite. In this case \cite[Lemma 4.9]{HLR}, 
$G$ contains a subgroup generated by two reflections $\rr ,\rr '$ that is already infinite. 
Write $H_u,H_{u'}$
for the corresponding hyperplanes and  note that $u'\notin H_u$. But then the rotation
$\rr \rr '$ has an angle that is an  irrational multiple of  $\pi$,
so that the orbit of $H_u$ under $\rr \rr  '$
is dense in $\R^d$. As $\widehat{\mu}=0$ on $H_u$ and
$|\widehat{\mu}|$ is invariant under $\rr \rr '$,
it follows that $\widehat{\mu}=0$ on a dense set. As $\widehat{\mu}$ is continuous, $\widehat{\mu}=0$ on
$\R^d$ thus $\mu=0$. Consequently, $(S, H_1 \cup \dots \cup  H_N)$ is a \HUP . 

(ii)  Or $G$ is finite. But then, there exists a subset $W$ of $\R^d$ (a Weyl chamber) such that
$\{gW : g\in G\}$ is a covering of $\R^d$ and every $x\in\R^d$
determines a unique  $g$ such that $x\in gW$.
Now take any continuous  function $\ffi$  on $\S^{d-1}\cap W$ and extend $\ffi$ to $\S^{d-1}$
by  the following rule: if $x\in\S^{d-1}$ then there exists a unique $g\in G$ such that $gx\in W$.
Writing $g$ as  $g=\prod _{j=1}^M \rr _{k_j}$ with
$k_1,\ldots,k_M\in\{1,\ldots,N\}$, we then set  $\ffi(x)=(-1)^M\ffi(gx)$.
The corresponding   measure $\mu(x)=\ffi(x)\,\mbox{d}\sigma(x)$ is well
defined, non-zero, absolutely continuous with respect to surface
measure,  and  satisfies $\rr _k\,_*\mu=-\mu$ by construction. By Lemma~\ref{lem:fund}
$\widehat{\mu}=0$ on $H_1 \cup \dots  \cup H_N$. Thus $(S, H_1 \cup
\dots \cup  H_N)$ 
fails to be  a \HUP . 
\end{proof}

\begin{remark}
When $d=2$ we recover the theorem of Lev and Sj\"olin: let
$\ell_1,\ell_2$ be  two lines that intersect at $0$. 
Then $(\S^1,\ell_1\cup\ell_2)$ is a \HUP ,  if and only if the angle between $\ell_1$ and $\ell_2$ is not in $\Q\pi$.

However, when $d\geq 3$, there exists sets of 3 hyperplanes such that  the angle between any two of them is
in $\Q\pi$ but such that the corresponding Coxeter group is infinite,
{\it see} \cite{RS} for a complete description of the $3$-dimensional case.
\end{remark}

\section*{Acknowledgments}
The authors kindly acknowledge financial support from the French ANR programs ANR
2011 BS01 007 01 (GeMeCod), ANR-12-BS01-0001 (Aventures).
This study has been carried out with financial support from the French State, managed
by the French National Research Agency (ANR) in the frame of the ”Investments for
the future” Programme IdEx Bordeaux - CPU (ANR-10-IDEX-03-02).
K.\ G.\ was  supported in part by the  project P26273 - N25  of the
Austrian Science Fund (FWF).
Both authors kindly acknowledge the support of the Austrian-French AMADEUS project 35598VB - ChargeDisq.

\end{document}